\documentclass{article}
\usepackage{maa-monthly-noheader}

\newtheorem*{theorem*}{Theorem}
\newtheorem{lemma}{Lemma}
\newtheorem*{conjecture*}{Conjecture}

\theoremstyle{remark}

\numberwithin{equation}{section}
\newcommand{\N}{\mathbb N}

\begin{document}
\title{Unconditional Prime-representing Functions, Following Mills}
\markright{Notes}
\author{Christian Elsholtz}
\maketitle


\date{}

\begin{abstract}
Mills proved that there exists a real constant $A>1$ such that
for all $n\in \N$ the values $\lfloor A^{3^n}\rfloor$ 
are prime numbers. 
No explicit value of $A$ is known, but assuming the Riemann
hypothesis one can choose $A= 1.3063778838\ldots .$
Here we give a first \emph{unconditional} variant:
$\lfloor A^{10^{10n}}\rfloor$ is prime, where
$A=1.00536773279814724017\ldots$ can be computed to millions of digits.
Similarly, $\lfloor A^{3^{13n}}\rfloor$ is prime, with
$A=3.8249998073439146171615551375\ldots .$ 

\end{abstract}


Mills \cite{Mills:1947} proved that there exists a real number $A>1$
such that for all $n \in \N$ the values $f(n)=\lfloor A^{3^n}\rfloor$ are prime.
For some related work see
\cite{CaldwellandCheng:2005, Kuipers:1950, Matomaki:2010,
Niven:1951, Ore:1952, Wright:1951, Wright:1954} 
 and 
\cite[Exercise 1.23]{EllisonandEllison:1985}.
Even though such formulae encode existing knowledge of primes, 
rather than generate new primes, and even though the proof shows that many
such values $A$ exist, it is quite astonishing
that not a single value $A$ is known. It is not even known 
whether any $A<10^{1\, 000\, 000}$ (say) exists
such that the statement holds, or not. 
The reason for this is that Mills made use of a result of Ingham that there
is always a prime between $n$ and $n+cn^{5/8}$, for some (nonexplicit) 
positive constant $c$, which implies that there always exists a prime
between any two \emph{sufficiently large} cubes. 
With current knowledge this is known only for cubes $t^3$ of size at least
$t^3\geq e^{e^{33.217}}$, a number which has $115,809,481,360,809$ digits 
(a result of Dudek, see \cite{Dudek:2016}). 
This number is very much larger than the
largest known primes, which in turn are primes of a very 
special type $q=2^n-1$ (Mersenne primes). 
Even with an improvement on Dudek's result and
with expected progress on primality tests, bridging this huge 
gap would appear to be several decades away.

Some escape routes out of this dilemma have been
studied:
\begin{enumerate}
\item
Caldwell and Cheng \cite{CaldwellandCheng:2005} 
observe that assuming the Riemann hypothesis the sequence of primes
$b_1=2, b_2=11, b_3=1361, b_4=2521008887, b_5=160222362 0400981813 1831320183$
can be continued such that Mills' result on $\lfloor A^{3^n} \rfloor$ 
holds with
$A= \lim_{n \rightarrow \infty}b_n^{3^{-n}}=1.3063778838\ldots .$ They also
write that an {\emph{unconditional}} value of $A$ is completely 
out of range of today's methods, 
due to the issue with ``sufficiently large'' cubes mentioned above.
\item For an unconditional result,
Wright \cite{Wright:1951} introduced 
a very rapidly increasing tower-type sequence: There exists a constant
$\omega= 1.9287800\ldots$ such that with 
$g_1=2^{\omega}, g_{n+1}=2^{g_n}$, the values of $\lfloor g_n \rfloor$ are
prime, for all integers $n\geq 1$,
starting with
$p_1=\lfloor 2^\omega\rfloor=3,p_2=\lfloor 2^{2^{\omega}}\rfloor =13,p_3=
\lfloor 2^{2^{2^{\omega}}}\rfloor =16381$.

\item Some formulae producing all primes also exist.
It follows from Wilson's theorem that the function 
\[f(n)=\left\lfloor \frac{n!\mod(n+1)}{n}\right\rfloor (n-1)+2\]
takes the value $f(n)=p_i$, where $p_i$ is the $i$th prime, when
$n=p_i-1$ and is $2$ otherwise.
Hence the values of $f$  are prime for all $n \in \N$.
Another example is Gandhi's formula
\[ p_n=\left\lfloor 1-\log_2\left(-\frac{1}{2}+ 
\sum_{d|P_{n-1}} \frac{\mu(d)}{2^d-1}\right) \right\rfloor,\]
where $\log_2$ denotes the logarithm to base $2$,
$\mu$ the M\"obius function,
$P_{n-1}= \prod_{i=1}^{n-1}p_i$, and $p_i$ is the $i$th prime in
ascending order.

Very recently a new formula was found \cite{Fridmanetal}: 
there exists a constant
$f_1=2.920050977316\ldots$ such that the sequence
$f_n=\lfloor f_{n-1}\rfloor (f_{n-1}-\lfloor f_{n-1}\rfloor +1)$ has the
property that $p_n=\lfloor f_n \rfloor$.

\end{enumerate}
A survey on such questions is in Ribenboim's book \cite{Ribenboim}.

In this note we prove the following \emph{unconditional} result 
on sequences in the spirit of Mills, which grow asymptotically 
much less rapidly than Wright's sequence.

\begin{theorem*}
\begin{itemize}
\item[a)]
Let $p$ be a Mersenne exponent, i.e., $2^p-1$ is a prime.
For every integer   
$m\geq  1\,438\,989$, there exists a real constant 
$A_{m,p}>1$ such that for all $n \in \N$ the values of all functions
$f_{m,p}(n)=\lfloor A_{m,p}^{m^n}\rfloor $ are prime.
Moreover, the values $A_{m,p}$ can be estimated as follows:
\[\frac{p}{m}\log 2
-\frac{2}{m2^{p}}<\log A_{m,p} <\frac{p}{m}\log 2. \]
If $p$ is large this gives a very high precision.
(The proof gives even more precise estimates.)
\item[b)]
Specializing to the Mersenne exponent $p=77\,232\,917$:
There is a constant $A=1.00536773279814724017\ldots$ 
such that all values of $\lfloor A^{10^{10n}}\rfloor, n \in \N$, are prime.
The constant $A$ can be computed to millions of decimal places.
\item[c)]
With the same $p$: $\lfloor A^{3^{13n}}\rfloor$ is prime
with $A=3.82499980734391461716\ldots$ .

\end{itemize}
\end{theorem*}
The proof makes use of two nontrivial ingredients. 

The first ingredient of the proof is an {\emph{explicit}} 
variant of the existence of primes in certain intervals.
Dudek \cite{Dudek:2016}
observed that for $m\geq 4.97117 \cdot 10^9$
there is a prime between $n^m$ and $(n+1)^m$, for {\emph{all}} 
values of $n \in \N$. 
The strongest currently known estimate of this kind
is due to Mattner \cite{Mattner}:
\begin{lemma}{\label{lem:mattner}}
Let $m\geq 1\, 438\, 989$. Then there is a prime with
$n^m< p < (n+1)^m$ for all $n\geq 1$.
\end{lemma}
The second ingredient is
a quite large prime. We choose the second-largest prime that is currently 
known, a Mersenne prime, see \cite{Mersenne}:
\begin{lemma}
$2^{77\,232\,917}-1$ is a prime number.
\end{lemma}

\begin{proof}[Proof of Theorem]
For our application we need to reduce the size of the interval 
in Lemma \ref{lem:mattner} by one element,
namely $(n+1)^m-1$ is divisible by $n$ so that the prime satisfies
$n^m< p < (n+1)^m-1$.
From this we can construct 
a sequence of primes $p_1, p_2, \ldots$ with 
$p_n^m<p_{n+1}<(p_n+1)^{m}-1$. 
Raising these inequalities (adding $1$ where necessary) to the
$m^{-n-1}$th power 
gives 
\[ p_n^{m^{-n}}<p_{n+1}^{m^{-n-1}}< (p_{n+1}+1)^{m^{-n-1}}
<(p_{n}+1)^{m^{-n}}.
\]
From this we see that the sequence 
$\alpha_n=(p_n^{m^{-n}})$ is an increasing sequence,
whereas the sequence $\beta_n=((p_{n}+1)^{m^{-n}})$ is decreasing
with increasing $n$. 
Hence the sequence $\alpha_n$ is also bounded and therefore 
the limit $A:=\lim_{n \rightarrow \infty} p_n^{m^{-n}}$ exists. It follows that
$p_n\leq A^{m^n}<p_n+1$ and so $p_n=\lfloor A^{m^n}\rfloor$.

We take $p_1=\lfloor A^{m^1}\rfloor=2^{77\,232\, 917}-1$. 
\[2^{77\,232\, 917}-1=2^{77\, 232\, 917}(1-\frac{1}{2^{77\, 232\,917}})
<A^{m}<2^{77\,232\,917}.\]
Taking the natural logarithm and observing that for small $x>0$ 
a simple explicit Taylor estimate gives
$-x-x^2< \log(1-x)< -x-\frac{x^2}{2}$ we find that
\begin{alignat}{3}
&\frac{77232917}{m}\log 2-\frac{1}{m2^{77232917}}-\frac{1}{m4^{77232917}}&\nonumber\\
<\log A <&\frac{77232917}{m}\log 2-\frac{1}{m2^{77232917}}
-\frac{1}{2m4^{77232917}}&\nonumber .
\end{alignat}
(This is more precise than stated in the theorem. Higher order 
Taylor estimates are also possible.)

Note that from this one can evaluate 
$\log A$ and therefore $A$ with an accuracy of millions of digits.
In particular, if $m=10^{10}$, then
$A=1.00536773279814724017\ldots,$ and similarly for
$m=3^{13}>1\, 438\, 989$.
(To see the implication that knowledge on high precision of $\log A$ 
implies also high precision for $A$:
Let $A_1<A_2$ be two constants with
$A_2=A_1(1+\varepsilon)$ (say), where $\varepsilon>0$ is a small constant. 
Then $\frac{\varepsilon}{2}<\log (1+\varepsilon)
=\log A_2- \log A_1=\log (1+\varepsilon)=
\varepsilon -\frac{\varepsilon^2}{2}\pm \cdots<\varepsilon$
and 
\[A_2-A_1=A_1 \varepsilon <2A_1 (\log A_2- \log A_1).\]
In other words, when $\log A$ is known with high precision and $A$ is of size,
say, $1<A<10$ as in the theorem, then very small deviations 
of $A$ to $A_1$ or $A_2$ would give
both $\log A_2-\log A_1$ and $A_2-A_1$ with about the same precision.
Hence $A$ is also known with high precision.)

\end{proof}

Similarly, every reader can produce their own formula by choosing a 
large number $m$ and a quite large prime $q$, 
for example the largest currently known prime $q=2^{82,589,933}-1$.
Then
$\frac{\log q}{m}- \frac{2}{mq} < \log A <
\frac{\log q}{m}$ determines $A$ with high precision.

Finally, as the actual distribution of primes might be much better than what 
can currently be proved, and based on some experiments, 
we conjecture the following:
\begin{conjecture*}
There is a constant $A$ (possibly near
$1.1966746500705764022$) such that
$f(n)=\lfloor A^{(n+1)^2}\rfloor $ is prime for all $n \geq 1$.
\end{conjecture*}
Note that the exponent $(n+1)^2$ grows polynomially compared to the
exponential growth of $m^n$ (for fixed $m$) in Mills-type  examples.
The value above would give 
$p_1=2, p_3=5, p_4=17, p_5=89, p_6=641, p_7=6619, p_8=97829, p_9=2070443$.

\begin{acknowledgment}{Acknowledgments.}
The author would like to thank Timothy Trudgian for drawing
the attention to his student Caitlin Mattner's work, 
and to the referees and editor for useful comments on the manuscript.\\
The author was partially supported by the Austrian Science Fund (FWF): W1230.\\
\end{acknowledgment}

\begin{affil}
Institut f\"ur Analysis und Zahlentheorie,
Kopernikusgasse 24, 8010 Graz, Austria\\ 
elsholtz@math.tugraz.at\\

\end{affil}

\end{document}